\documentclass[preprint,11pt,a4paper]{elsarticle}
\usepackage{amsmath,amsfonts,amssymb,amsthm,graphics,color}

\theoremstyle{plain}
\newtheorem{theorem}{Theorem}

\newtheorem{proposition}[theorem]{Proposition}

\theoremstyle{definition}
\newtheorem{assumption}{Assumption}
\newtheorem{example}{Example}
\newtheorem*{remark}{Remark}

\usepackage{xcolor}

\journal{Applied Mathematics and Computation}
\date{October 24, 2025}

\begin{document}
\title{Exponential integrators for parabolic problems with non-homogeneous boundary conditions}

\author[uva]{Carlos Arranz-Sim\'{o}n\corref{cor1}}
\ead{carlos.arranz@uva.es}

\author[uibk,disc]{Alexander Ostermann}
\ead{alexander.ostermann@uibk.ac.at}

\cortext[cor1]{Corresponding author}
\address[uva]{IMUVA-Mathematics Research Institute, Facultad de Ciencias, University of Valladolid, 47011 Valladolid, Spain}
\address[uibk]{Department of Mathematics, University of Innsbruck, 6020 Innsbruck, Austria}
\address[disc]{ Digital Science Center, University of Innsbruck, 6020 Innsbruck, Austria}

\begin{abstract}
Exponential Runge--Kutta methods are a well-established tool for the numerical integration of parabolic evolution equations. However, these schemes are typically developed under the assumption of homogeneous boundary conditions. In this paper, we extend classical convergence results to the case of non-homogeneous boundary conditions. Since non-homogeneous boundary  conditions typically cause order reduction, we introduce a correction strategy based on smooth extensions of the boundary data. This results in a reformulation as a homogeneous problem with modified source term, to which standard exponential integrators can be applied. For linear problems, we prove that the corrected schemes recover the expected convergence order, and that higher orders can be attained with suitable quadrature rules, reaching order $2s$ for $s$-stage Gauss collocation methods. For semilinear problems, our approach preserves the convergence orders guaranteed by exponential Runge--Kutta methods satisfying the corresponding stiff order conditions. Numerical experiments validate the theoretical findings.
\end{abstract}

\maketitle

\section{Introduction}

Exponential Runge--Kutta integrators have proven to be very efficient methods for solving stiff differential equations, a challenging topic in numerical analysis. These methods have been used to solve various physical problems, including magnetohydrodynamic equations, models for meteorology, and advection-reaction-diffusion equations (see, e.g., \cite{Magnet,reaction,Meteo}). Such problems are typically formulated as abstract evolution equations in a Banach space, in the framework of semigroup theory in a Banach space \cite{Henry,Pazy}.

However, the exponential methods derived in \cite{HO05, HO05e} require the use of periodic or homogeneous boundary conditions, which is a strong limitation for practical purposes. It is well known that applying Runge--Kutta or exponential Runge--Kutta methods to time-integrate a problem with non-homogeneous boundary conditions generally reduces the order of convergence of the methods. In a series of papers \cite{CaM25, CaM24, Ca22, CaM18}, modifications to the exponential Runge--Kutta methods where proposed to alleviate or avoid the order reduction. This work is another contribution in this direction.

To integrate a problem with non-homogeneous boundary conditions, we use a smooth function to extend the boundary data to the interior of the corresponding domain. In \cite{EMO18} different exam\-ples of such extensions are provided in the context of splitting methods. Next, we consider the evolution problem satisfied by the difference of the original solution and the extension, which turns out to be homogeneous. At this point, well-known exponential Runge--Kutta schemes can be applied to the latter,  though additional corrections to the source term of the new equation are still necessary. We discuss how different choices of the extension lead to simpler source terms in the new problem.

We briefly summarize the content of the paper. In Section 2, we examine the integration of linear initial boundary value problems using the proposed strategy. After performing the correction, we construct the schemes using exponential quadrature formulas similar to those of \cite{HO05}. Firstly, we find that the proposed methods have the expected convergence order due to the accuracy of the underlying Runge--Kutta methods. Subsequently, we observe that some quadrature rules with additional properties give rise to schemes with higher convergence orders. We also discuss the feasibility of using them with variable time steps under different hypotheses on the step size sequence.

In Section~3, we propose particular corrections that lead to an increase of the convergence order of linear problems up to classical order. Selecting a correction that satisfies a specific set of boundary conditions allows us to generalize the argument demonstrating an additional increase in the order of convergence when the underlying quadrature rule admits it. In particular, for Gauss collocation methods, this means raising the order from $s$ to order $2s$.

Finally, in Section~4 we address semilinear problems. We demonstrate that our approach yields the same order of convergence as obtained for homogeneous problems (see \cite{HO05}), where the methods fulfil the so-called stiff order conditions that guarantee a certain order of convergence.

Along the corresponding sections, we show many numerical examples that illustrate the theoretical results: both considering different corrections and exponential methods.

\section{Linear problems}\label{sec:lin-prob}

In this section we are concerned with the numerical time integration of linear parabolic boundary value problems of the form
\begin{align}\label{linearproblem}
&\left\{
\begin{aligned}
\partial_t u & = Du + f(t), \\
B u & = b(t), \\
u(0) & = u_0,
\end{aligned}
\right.
\end{align}
with $u = u(t)$ for $t \in [0,T]$. Note that all functions in (\ref{linearproblem}) may depend on a spatial variable $x$ in a domain $\Omega \subset \mathbb{R}^d$, e.g., $u = u(t,x)$. However, we usually suppress this dependence in the notation. We assume that $D$ is a strongly elliptic differential operator with smooth coefficients (for example, the Laplacian), $B$ is a boundary operator representing Dirichlet, Neumann or mixed boundary conditions, $u_0$ is the initial value, $f$ is the time-dependent source term and the function $b$ provides the prescribed boundary values. Note that, in general, $b$ is allowed to depend on time.

Our intention is to reduce the original problem to one with homogeneous boundary conditions. This allows us to formulate the problem as an abstract evolution problem and apply exponential methods. With this idea in mind we denote by $z = z(t,x)$ a smooth function that satisfies the prescribed boundary data
\begin{equation}\label{zbc}
Bz = b.
\end{equation}
There are many efficient ways to obtain such functions. For some possible choices in a similar context, see \cite{EMO18}. One option is the harmonic extension of the boundary data. In this case, $z$ solves the elliptic problem
\begin{equation}\label{elliptic}
Dz = 0, \quad Bz = b.
\end{equation}
Since we are only interested in a function that satisfies the boundary conditions, we can even replace the differential operator $D$ with another appropriate operator (e.g., the Laplacian $\Delta$), for which the problem is simpler or can be solved more efficiently. A second option is using $z$ as a solution of the parabolic evolution problem
\begin{equation}
z_t = Dz, \quad Bz = b.
\end{equation}
Although this approach is more difficult to implement in practice, it leads to simpler corrected equations, as we will show in the following paragraph. An idea for the third strategy is to employ the very efficient spatial filtering technique used for noise reduction (or smoothing) in linear image processing described, for instance, in \cite{EMO18}. This consist of using an iterative procedure of weighted averages to construct a smooth function that satisfies (\ref{zbc}), starting from the boundary values. Finally, we point out that, for time-independent boundary conditions, we can simply set $z_n = u_n$ as the correction, since the numerical solution satisfies the prescribed boundary data.

Having a smooth extension $z$ of the boundary data at hand, we consider $w = u - z$, which is the solution of the equation
\begin{align}
&\left\{
\begin{aligned}
\partial_t w & = Dw + f(t) + k(t), \\
B w & = 0, \\
w(0) & = u_0 - z(0),
\end{aligned}
\right.
\label{linearproblemcorrected}
\end{align}
where $k = Dz - \partial_t z$. For the harmonic extension, this simplifies to $k  = - \partial_t z$.

The previous problem, now with homogeneous boundary conditions, can be formulated as an abstract evolution equation in a Banach space $\left(X, \| \cdot \| \right)$. We denote by $A$ the restriction of $D$ to the vector space $Bu = 0$, whose domain is denoted by $D(A)$. Our basic assumptions on this operator are those of \cite{Henry, Lunardi, Pazy}.

\begin{assumption}\label{ass:sg}\rm
Let $A : D(A) \rightarrow X$ be sectorial, i.e., $A$ is a densely defined and closed operator on $X$ satisfying the resolvent condition
\begin{equation}
\left\| (\lambda I - A)^{-1} \right\| \leq \dfrac{M}{\left|\lambda - a\right|},
\end{equation}
on the sector $\left\lbrace \lambda \in \mathbb{C};\ 0 \leq \left| \text{arg}\left(\lambda - a \right)\right| \leq \theta,\ \lambda \neq a \right\rbrace$ for $M \geq 1$, $a \in \mathbb{R}$, and $\pi/2 < \theta < \pi $.
\end{assumption}

Under this assumption, the operator $A$ is the infinitesimal generator of an analytic semigroup $\lbrace \text{e}^{tA}\rbrace_{t\geq 0}$. For $\omega > a$, the fractional powers of $\widetilde{A} = \omega I - A$ are well-defined. We set, for $ 0 \leq \alpha \leq 1$, the spaces $X_{\alpha} = D(\widetilde{A}^{\alpha})$ endowed with the norms $\|u\|_{\alpha} = \|\widetilde{A}^{\alpha} u \|$. Note that this definition does not depend on $\omega$, since different choices of $\omega$ lead to equivalent norms. Given a function $u: [0,T] \rightarrow X_{\alpha}$, we also define the norm
\begin{equation}
\|u\|_{\infty} = \max_{0 \leq t \leq T} \|u(t)\|_{\alpha}.
\end{equation}
Under Assumption~\ref{ass:sg}, the following stability bound holds uniformly on $0 \leq t \leq T$,
\begin{equation}\label{stability1}
\|\text{e}^{tA} \| + \|t^{\gamma} \widetilde{A}^{\gamma}\text{e}^{tA}\| \leq C, \quad \gamma > 0.
\end{equation}
Employing this framework, the problem (\ref{linearproblemcorrected}) can be written in abstract form
\begin{align}\label{linearproblemasbtract}
&\left\{
\begin{aligned}
w'(t) & = Aw(t) + f(t) + k(t), \\
w(0) & = w_0.
\end{aligned}
\right.
\end{align}
For readers who are unfamiliar with this formulation, we provide two illustrative examples.

\begin{enumerate}
\item[(i)] Let $X = L^p(\Omega)$, $1 \leq p < \infty$, $D$ be the Laplace operator and set Dirichlet boundary conditions, $Bu = u |_{\partial \Omega}$. Then $A$ generates an analytic semigroup of operators in $X \cap \ker B$ whose domain is $D(A) = W^{2,p}(\Omega) \cap W^{1,p}_0(\Omega)$.

\item[(ii)] Let $\Omega = (0,1) \subset \mathbb{R}$ and consider as $D$ the one-dimensional Laplace operator, but now discretized by centered second-order finite differences with $N$ inner grid points and spatial grid size $\Delta x = 1/(N+1)$. For Dirichlet boundary conditions we have $X = \mathbb{R}^N$, and $A$ is represented by the $N\times N$ matrix
\begin{equation}
A = \dfrac{1}{\left(\Delta x\right)^2} \begin{pmatrix}
-2 & 1 &  &  0\\
1 & -2 & 1 &  \\
 & \ddots & \ddots & \ddots \\
0 &  & 1 & -2
\end{pmatrix}.
\end{equation}
In this case, the operators of the semigroup coincide with the matrix exponentials of $tA$.
\end{enumerate}

We use this semigroup framework because is very powerful for proving convergence and includes both abstract PDEs and their spatial discretizations.

To integrate problem \eqref{linearproblemcorrected}, we use an exponential integrator based on an $s$-stage exponential quadrature rule. Such a rule is given by $s$ nodes $c_1, \dots, c_s$ and weights $b_1(\tau A), \dots, b_s(\tau A)$ (see, e.g., \cite{HO05}). We assume that the well-known order conditions for order $s$ to hold:
\begin{equation}\label{orderconditions1}
\frac{1}{(j-1)!}\sum_{i=1}^s b_i(\tau A)c_i^{j-1}  =  \varphi_j(\tau A), \quad j = 1, \ldots, s,
\end{equation}
where $\varphi_j(z)$ for $j\ge 1$ are the functions given by the following recurrence relation
\begin{equation*}
\varphi_j(z) = \dfrac{\varphi_{j-1}(z) - \frac1{(j-1)!}}{z},\qquad \varphi_{0}(z) = \text{e}^z.
\end{equation*}
These so-called $\varphi$ functions are bounded on the spectrum of every $A$ satisfying Assumption~\ref{ass:sg}, so the operators $\varphi(\tau A)$ are well-defined and bounded for $\tau > 0$. Moreover, the following integral representation is valid
\begin{equation}\label{eq:int-repr}
\varphi_j(\tau A) = \dfrac{1}{\tau^j} \int_0^{\tau} \text{e}^{\left(\xi - \tau\right)A}\dfrac{\xi^{j-1}}{\left(j-1\right)!} \, d\xi,
\end{equation}
which will be useful in our analysis. For non-confluent nodes the weights are uniquely determined in terms of the nodes and the $\varphi$ functions, since \eqref{orderconditions1} gives rise to a Vandermonde system for the weights. This shows hat the weight functions are linear combination of the $\varphi$ functions, and due to~\eqref{stability1} and~\eqref{eq:int-repr}, the following stability bounds hold uniformly for $0\le t\le T$
\begin{equation}\label{stability-weights}
\|b_i(tA) \| \leq C, \quad 1\le i \le s.
\end{equation}

For the time integration of \eqref{linearproblemasbtract}, we choose a step size sequence $\left\lbrace \tau_n \right\rbrace_{n=0}^{N-1}$ and define the temporal grid points
\begin{equation*}
t_0 = 0, \quad t_{n+1} = t_n + \tau_n, \ n = 0, \dots, N-1, \quad t_N = T.
\end{equation*}
Then, the numerical approximation $w_{n+1}$ to the exact solution $w(t)$ at time $t=t_{n+1}$ is given by
\begin{equation}\label{wnumlinearscheme}
w_{n+1} = \text{e}^{\tau_n A} w_n + \tau_n \sum_{i=1}^s b_i(\tau_n A ) \bigl(f\left(t_n + \tau_n c_i\right) + k\left(t_n + \tau_n c_i\right)\bigr).
\end{equation}
The latter can be rearranged so as to be written in terms of $u$,
\begin{equation}\label{unumlinearscheme}
u_{n+1} = \text{e}^{\tau_n A} \left(u_n - z(t_n)\right) + z(t_{n+1}) + \tau_n \sum_{i=1}^s b_i\left(\tau_n A \right) \bigl(f\left(t_n + \tau_n c_i\right) + k\left(t_n + \tau_n c_i\right)\bigr).
\end{equation}

Now we are in a position to state our first convergence result.

\begin{theorem} \label{thm:order-s}
Consider the exponential Runge--Kutta method (\ref{unumlinearscheme}) for the numerical solution of~(\ref{linearproblem}). If $f^{(s)} + k^{(s)} \in L^1\left(0,T;X\right)$, then the method converges with order $s$ and the error satisfies
\begin{equation*}
\|u(t_n) - u_n \| \leq C \sum_{j=0}^{n-1} \tau_{j}^s \int_{t_j}^{t_{j+1}} \| f^{(s)}(\tau) + k^{(s)}(\tau) \| \, d\tau
\end{equation*}
for $0 \leq t_n \leq T$. The constant $C$ depends on $T$, but is independent of $n$ and $\left\lbrace \tau_n \right\rbrace_{n=0}^{N-1}$.
\end{theorem}

\begin{proof}
The proof follows very closely that of Theorem 1 in \cite{HO05}. The main difference is just that we are using here variable step sizes. As this is only a small modification, it would not fully justify to give the proof again. However, as some of the formulas appearing in the proof are also required later, we decided to present the proof.

Expanding the exact solution of (\ref{linearproblemasbtract}) into a Taylor series gives
\begin{align}\label{exactTaylor}
w(t_{n+1}) & = \text{e}^{\tau_n A} w(t_n) + \int_0^{\tau_n} \text{e}^{\left(\tau_n - \xi\right)A} \left(f\left(t_n + \xi\right) + k\left(t_n + \xi\right)\right)\, d\xi \notag  \\
& = \text{e}^{\tau_n A} w(t_n) + \sum_{\ell=0}^{s-1} \varphi_{\ell+1}(\tau_{n}A) \tau_{n}^{\ell+1} \left(f^{(\ell)}\left(t_n\right) + k^{(\ell)}\left(t_n\right)\right) \notag \\
& \qquad + \int_0^{\tau_n}\text{e}^{\left(\tau_n - \xi\right)A} \int_0^{\xi} \dfrac{(\tau_n - \mu)^{s-1}}{\left(s-1\right)!} \left(f^{(s)}\left(t_n + \mu\right) + k^{(s)}\left(t_n + \mu\right)\right) d\mu \, d\xi.
\end{align}
On the other hand, the Taylor series of the numerical solution \eqref{wnumlinearscheme} is given by
\begin{align}\label{numTaylor}
w_{n+1} & = \text{e}^{\tau_n A} w_n + \tau_n \sum_{i=1}^s b_i(\tau_n A) \sum_{\ell = 0}^{s-1} \dfrac{\tau_n^{\ell} c_i^{\ell}}{\ell !} \left(f^{(\ell)}\left(t_n\right) + k^{(\ell)}\left(t_n\right)\right)\notag  \\
& \qquad + \tau_n \sum_{i=1}^s b_i(\tau_n A) \int_0^{c_i \tau_n} \dfrac{(c_i \tau_n - \mu)^{s-1}}{(s-1)!}\left(f^{(s)}\left(t_n + \mu\right) + k^{(s)}\left(t_n + \mu\right)\right) d\mu.
\end{align}
Let $e_n = u(t_n) - u_n$ denote the difference between the exact and the numerical solution. Note that the order conditions \eqref{orderconditions1} guarantee that
\begin{equation}\label{recursion}
e_{n+1} = \text{e}^{\tau_n A}e_n + \delta_{n+1},
\end{equation}
where the local errors $\delta_{n+1}$ are given by
\begin{equation}\label{localerror}
\begin{aligned}
\delta_{n+1} = & \int_0^{\tau_n} \text{e}^{\left(\tau_n - \xi\right)A} \int_0^{\xi} \dfrac{(\tau_n - \mu)^{s-1}}{\left(s-1\right)!}
\left(f^{(s)}(t_n + \mu) + k^{(s)}(t_n + \mu)\right) d\mu \, d\xi \\
& - \tau_n \sum_{i=1}^s b_i(\tau_n A)\int_0^{c_i \tau_n} \dfrac{(c_i \tau_n - \mu)^{s-1}}{(s-1)!}
\left(f^{(s)}(t_n + \mu) + k^{(s)}(t_n + \mu)\right) d\mu.
\end{aligned}
\end{equation}
Using the bounds on the semigroup \eqref{stability1} and the weight functions \eqref{stability-weights} shows that
\begin{equation}
\|\delta_{n+1} \| \leq \tau_{n}^{s}\int_{t_{n}}^{t_{n+1}} \|f^{(s)}(\xi) + k^{(s)}(\xi)\|\,d\xi.
\end{equation}
Since $e_0 = 0$, the recursion (\ref{recursion}) has the solution
\begin{equation*}
e_n = \sum_{j=1}^n \text{e}^{\left(t_n-t_j\right)A} \delta_j.
\end{equation*}
The desired estimate now follows after bounding the latter by
\begin{equation}
\|e_n\| \leq \sum_{j=1}^n \|\text{e}^{\left(t_n-t_j\right)A}\| \   \|\delta_j\|,
\end{equation}
using the stability bound \eqref{stability1} and the fact that $u_n = w_n + z(t_n)$.
\end{proof}

In practice some methods show higher order of convergence. This is due to the fact that Theorem~\ref{thm:order-s} is not optimal for methods whose underlying quadrature rule is of order $s+1$, that is,
\begin{equation}\label{eq:order-s+1}
\sum_{i=1}^s b_i(0) c_i^{j-1} = \dfrac{1}{j},\quad 1\le j \le s+1.
\end{equation}
This is the case, for example, for all Gauss methods and the Radau methods with $s \geq 2$. For the following theorem, we need an additional assumption on the sequence of step lengths. We require that
\begin{equation}\label{stepsizegrow}
\tau_{j} \leq \kappa \tau_{j+1}, \quad 0 \leq j \leq N-1
\end{equation}
for some $\kappa >0$. Quasi-uniform step size sequences, e.g., satisfy this assumption.
\begin{theorem}\label{thm:order-s+1}
Under the hypothesis of Theorem~\ref{thm:order-s} and the additional condition \eqref{eq:order-s+1}, we consider the application of the scheme \eqref{unumlinearscheme} to solve problem \eqref{linearproblem}. If the step size sequence satisfies \eqref{stepsizegrow} and $f^{(s+1)} + k^{(s+1)} \in L^1\left(0,T;X\right)$, then
\begin{align*}
\|u_n - u(t_n)\| & \leq C \sum_{j=0}^{n-1}\tau_j^{s+1}\int_{t_j}^{t_{j+1}} \|f^{(s+1)}(\xi) + k^{(s+1)}(\xi)\| \, d\xi  \notag  \\
& \qquad + C\tau_{n-1}^{s+1} \|f^{(s)} + k^{(s)}\|_{\infty} + C\kappa \sum_{j=0}^{n-2} \tau_j^{s+1} \left(\left(t_n - t_{j+1}\right)^{-1}\tau_j \right)\, \|f^{(s)} + k^{(s)}\|_{\infty}
\end{align*}
is satisfied for $0 \leq t_n \leq T$. The constant $C$ depends on $T$, but is independent of $n$ and $\left\lbrace \tau_n \right\rbrace_{n=0}^{N-1}$.
\end{theorem}

\begin{remark} \rm
Before we prove this theorem, we will discuss what the above estimate means in terms of order of convergence. Note that the following (Riemann) sum is bounded by the corresponding integral,
\begin{align*}
\sum_{j=0}^{n-2} (t_n - t_{j+1})^{-1}\tau_j & \leq \kappa + \kappa\sum_{j=0}^{n-3}(t_n - t_{j+1})^{-1}\tau_{j+1} \\
& \leq \kappa + \kappa \int_0^{t_{n-1}} (t_n - s)^{-1} ds \leq \kappa + M \left| \log \tau_{n-1}\right|,
\end{align*}
where $M$ may depend on $T$, but not on $n$ and $\left\lbrace \tau_n \right\rbrace_{n=0}^{N-1}$. This bound of the sum shows that the method has order $s+1$ up to a logarithmic factor $\left| \log \tau_{n-1}\right|$. To clarify the meaning of the theorem, we detail some typical situations that appear in practice.
\begin{enumerate}

\item[(a)]
$\left\lbrace \tau_n \right\rbrace_{n=0}^{N-1}$ constant, $\tau_j = \tau$. In this case we get for $g=f+k$ the following bound
\begin{equation*}
C\tau_{n-1}^{s+1} \|g^{(s)}\|_{\infty} + C\kappa \sum_{j=0}^{n-2} \tau_j^{s+1} \left(\left(t_n - t_{j+1}\right)^{-1}\tau_j \right)\, \|g^{(s)}\|_{\infty} \leq C\tau^{s+1} \left(1 + \left| \log \tau \right| \right) \|g^{(s)}\|_{\infty},
\end{equation*}
showing the order $s+1$ apart from the logarithmic factor $\left| \log \tau \right|$, which is negligible compared to $\tau^{s+1}$. Actually, in the case of constant step sizes a different proof (see \cite[Thm.~2]{HO05}) shows that the logarithmic factor can be dropped. This is because the summation-by-parts argument can be applied there, whereas it is not applicable in cases involving variable step sizes.

\item[(b)]
 $\left\lbrace \tau_n \right\rbrace_{n=0}^{N-1}$ quasi-uniform, $\alpha \tau \leq \tau_j \leq \tau$. In this case we get for $g=f+k$ the bound
\begin{equation*}
C\tau_{n-1}^{s+1} \|g^{(s)}\|_{\infty} + C\kappa \sum_{j=0}^{n-2} \tau_j^{s+1} \left(\left(t_n - t_{j+1}\right)^{-1}\tau_j \right)\, \|g^{(s)}\|_{\infty} \leq C\tau^{s+1} \left(1 + \left| \log \tau \right| \right) \|g^{(s)}\|_{\infty},
\end{equation*}
with a constant depending on $\alpha$.

\item[(c)]
Under the stronger assumption $\widetilde{A}^{\beta} \left(f^{(s)} + k^{(s)}\right) \in L^{\infty}\left(0,T; X\right)$ for some $\beta>0$, we can drop the logarithmic factor by slightly modifying the proof. Using the bound
\begin{align*}
\left\|\text{e}^{\left(t_n - t_j\right)A} A \left(f^{(s)} + k^{(s)}\right)\right\| & = \left\|\text{e}^{\left(t_n - t_j\right)A} \widetilde{A}^{1-\beta} \left(A \widetilde{A}^{-1}\right)\widetilde{A}^{\beta} \left(f^{(s)} + k^{(s)}\right)\right\| \\
& \leq C \left(t_n - t_j\right)^{\beta-1} \left\|\widetilde{A}^{\beta} \left(f^{(s)} + k^{(s)}\right)\right\|
\end{align*}
and the same reasoning as in part~(a) we get an integrable kernel $\int_0^{t_n} (t_n-\xi)^{\beta-1}d\xi < \infty$. This leads to the bound
\begin{equation*}
C\tau_{n-1}^{s+1} \|g^{(s)}\|_{\infty} + C\kappa \sum_{j=0}^{n-2} \tau_j^{s+1} \left(\left(t_n - t_{j+1}\right)^{\beta-1}\tau_j \right)\, \|g^{(s)}\|_{\infty} \leq C\, \tau^{s+1}\,  \|\widetilde{A}^{\beta} g^{(s)}\|_{\infty},
\end{equation*}
with $\tau = \max_{0 \leq j \leq N-1} \tau_j$.
\end{enumerate}
\end{remark}

\begin{proof}[Proof of Theorem \ref{thm:order-s+1}]
We start our reasoning in the same way as in the proof of Theorem~\ref{thm:order-s}. However, since the order conditions~\eqref{orderconditions1} are only satisfied up to order $s$, we expand the local error~\eqref{localerror} one order further to get
\begin{equation}\label{eq:refined-le}
\delta_{n+1} = \tau_{n}^{s+1} \, \psi_{s+1}(\tau_{n} A)\left(f^{(s)}(t_n) + k^{(s)}(t_n)\right) + \widetilde{\delta}_{n+1},
\end{equation}
where
\begin{equation*}
\|\widetilde{\delta}_{n+1} \| \leq \tau_{n}^{s+1}\int_{t_{n}}^{t_{n+1}} \|f^{(s+1)}(\xi) + k^{(s+1)}(\xi)\|\,d\xi
\end{equation*}
and
\begin{equation*}
\psi_{s+1}\left(\tau_n A\right) = \varphi_{s+1}\left(\tau_n A\right) - \dfrac{1}{s!}\sum_{i=1}^s b_i\left(\tau_n A \right) c_i^s.
\end{equation*}
We can bound the second term in~\eqref{eq:refined-le} by reasoning as in the previous theorem
\begin{equation*}
\sum_{j=0}^{n-1} \|\text{e}^{t_n - t_{j+1}}\| \, \|\widetilde{\delta}_j \| \leq C \sum_{j=0}^{n-1}\tau_j^{s+1}\int_{t_j}^{t_{j+1}} \|f^{(s+1)}(\xi) + k^{(s+1)}(\xi)\| \, d\xi.
\end{equation*}
For the first term in~\eqref{eq:refined-le}, we recall that $f^{(s+1)} \in L^1\left(0,T;X\right)$ implies that $f^{(s)} \in L^{\infty}\left(0,T;X\right)$. Since $\psi_{s+1}\left(\tau A\right)$ is uniformly bounded in $0 \leq \tau \leq T$, the term with index $j = n - 1$ satisfies
\begin{equation*}
\left\| \tau_{n-1}^{s+1} \,\psi_{s+1}(\tau_{n-1} A)  \left(f^{(s)}(t_{n-1}) + k^{(s)}(t_{n-1})\right)\right\| \leq C \tau_{n-1}^{s+1} \bigl\|f^{(s)}+k^{(s)}\bigr\|_{\infty}.
\end{equation*}
For $0 \leq j \leq n-2$, we note that $\psi_{s+1}\left(0\right) = 0$ due to the additional order condition \eqref{orderconditions1}. Thus,
\begin{equation}
\psi_{s+1}(\tau_j A ) = \psi_{s+1}(\tau_j A ) - \psi_{s+1}(0) = \tau_j A\; \psi^{[1]}_{s+1}(\tau_j A )
\end{equation}
for a certain bounded operator $\psi^{[1]}_{s+1}(\tau_j A )$. We conclude the proof by taking this into account in the estimate
\begin{equation*}
\left\| \tau_j A \,\text{e}^{(t_n - t_{j+1})A}\;\psi_{s+1}^{[1]}(\tau_j A) \tau_{j}^{s+1} \left(f^{(s)}(t_{j}) + k^{(s)}(t_j)\right)\right\| \leq C \tau_j^{s+1} \left((t_n - t_{j+1})^{-1} \tau_j\right) \bigl\|f^{(s)}+k^{(s)}\bigr\|_{\infty},
\end{equation*}
where we have also used \eqref{stability1}.
\end{proof}

In practice, methods satisfying the \emph{additional} condition
\begin{equation}
\sum_{i=1}^s b_i(0) c_i^{s+1} = \dfrac{1}{s+2},
\label{order-s+2}
\end{equation}
which means that the underlying quadrature rule is of order $s+2$, show a slightly higher order of convergence whenever the source term has a higher spatial regularity. For the sake of simplicity, from now on we consider only constant step sizes. Variable step sizes, however, can be considered as well by taking into account the arguments used in this section. The following theorem is a direct adaptation of \cite[Thm.~3]{HO05} to our case with non-homogeneous boundary conditions.

\begin{theorem}\label{thm:beta}
Let $0 \leq \beta \leq 1$. Under the hypotheses of Theorem~\ref{thm:order-s+1} and the additional condition~\eqref{order-s+2}, we consider the application of the scheme \eqref{unumlinearscheme} to approximate the solution of \eqref{linearproblem}. If the step size sequence is constant, $f^{(s+2)} + k^{(s+2)} \in L^1\left(0,T;X\right)$ and $\widetilde{A}^{\beta}\left(f^{(s+1)} + k^{(s+1)}\right)\in L^1\left(0,T;X\right)$, then the method converges with order $s+1+\beta$ and the error satisfies the bound
\begin{align*}
\|u_n - u(t_n)\| & \leq C \tau^{s+1+\beta}\left(\left\|\widetilde{A}^{\beta}\bigl(f^{(s)}(0) + k^{(s)}(0)\bigr)\right\| + \int_0^{t_n} \left\|\widetilde{A}^{\beta}\bigl(f^{(s+1)}(\tau) + k^{(s+1)}(\tau)\bigr)\right\|d\tau \right)\\
& \qquad + C \tau^{s+2}\left(\left\|f^{(s+1)}(0) + k^{(s+1)}(0)\right\| + \int_0^{t_n} \left\|f^{(s+2)}(\tau) + k^{(s+2)}(\tau)\right\|d\tau \right) ,
\end{align*}
uniformly on $0\leq t_n \leq T$. The constant $C$ depends on $T$, but is independent of $n$ and $\tau$.
\end{theorem}

\begin{proof}
Since we are considering constant step sizes, it suffices to apply \cite[Thm.~3]{HO05} to \eqref{wnumlinearscheme} and take into account $u_n - u(t_n) = w_n - w(t_n)$.
\end{proof}

\section{More corrections and some numerical illustrations}\label{sec:more-corrections}

In this section, we present a strategy for improving the order of convergence by using more sophisticated corrections. For simplicity we only consider constant step sizes. However, the arguments presented in Section~\ref{sec:lin-prob} can be applied to consider variable step sizes as well.

Note that the order reduction in Theorems~\ref{thm:order-s+1} and~\ref{thm:beta} occurred primarily because certain time derivatives of the source term $g = f + k$ did not satisfy homogeneous boundary conditions. This motivates us to impose the (slightly stronger) condition $Bg=0$. A straightforward calculation shows that
\begin{equation*}
Bg = B(f + k) = Bf + BDz - B\partial_t z = Bf + BDz -b'.
\end{equation*}
Thus, $Bg=0$ is satisfied if $BDz = -Bf + b'$. To fulfil the latter condition, we consider a smooth function $z^{[1]}$ satisfying
\begin{equation}\label{z11}
Bz^{[1]}(t) = - Bf(t) + b'(t)
\end{equation}
and then employ the correction given by
\begin{equation}\label{z12}
Dz(t) = z^{[1]}(t), \quad Bz = b(t).
\end{equation}
Using this correction improves the order of convergence.

\begin{theorem}\label{thm:order-s+2}
Under the hypotheses of Theorem~2 and the additional condition \eqref{order-s+2}, we consider the application of the scheme \eqref{unumlinearscheme} with the correction \eqref{z11}, \eqref{z12} to approximate the solution of~\eqref{linearproblem}. Further assume that $f^{(s+2)} + k^{(s+2)} \in L^1\left(0,T;X\right)$ and $D\bigl(f^{(s+1)} + k^{(s+1)}\bigr)\in L^1(0,T;X)$. Then the method converges with order $s+2$ and the error satisfies the bound
\begin{align*}
\|u_n - u(t_n)\| & \leq C \tau^{s+2}\left(\left\|A\bigl(f^{(s)}(0) + k^{(s)}(0)\bigr)\right\| + \int_0^{t_n} \left\|A\bigl(f^{(s+1)}(\tau) + k^{(s+1)}(\tau)\bigr)\right\|d\tau \right. \notag  \\
& \qquad + \left. \left\|f^{(s+1)}(0) + k^{(s+1)}(0)\right\| + \int_0^{t_n} \left\|f^{(s+2)}(\tau) + k^{(s+2)}(\tau)\right\|d\tau \right) \notag ,
\end{align*}
uniformly on $0\leq t_n \leq T$. The constant $C$ depends on $T$, but is independent of $n$ and $\tau$.
\end{theorem}

\begin{proof}
Taking derivatives with respect to time in $Bg=0$ and noting that $D g^{(s+1)} \in  L^1(0,T;X)$ we deduce that $g^{(s+1)} \in D(A)$ and therefore that $Ag^{(s+1)} \in L^1(0,T;X)$. Now, the application of \cite[Thm.~3]{HO05} concludes the proof.
\end{proof}

Further corrections can be made in this spirit to achieve higher temporal order of convergence. The remark following Theorem~3 in \cite{HO05} asserts that the condition $\beta \leq 1$ in this theorem is just made for simplicity, and that the order may be improved when additional conditions are met. In particular, full (classical) order is achieved with a sufficiently smooth source term and periodic boundary conditions.

With regard to our scheme, we assume the order conditions
\begin{equation}\label{eq:class-oc}
\sum_{i=1}^s b_i(0) c_i^{s+l} = \dfrac{1}{s+l+1},\quad 0\le l\le m-1,
\end{equation}
which are satisfied, e.g., for Gauss methods with $m\le s$, and for Radau methods with $m\le s-1$. Further, let $z^{[m]}$ be a smooth function satisfying
\begin{equation}\label{eq:cond-highest}
Bz^{[m]}(t) = B\bigl(-D^{m-1}f(t)\bigr) + b'_{m-1}(t),
\end{equation}
and define recursively the functions $z^{[l]}$, $l = m-1, \ldots, 1$,
\begin{equation}\label{eq:cond-recursion}
Dz^{[l]}(t) = z^{(l+1)}(t), \quad Bz^{[l]}(t) = B\bigl(-D^{l-1}f(t)\bigr) + b'_{l-1}(t) =: b_l(t),
\end{equation}
and finally $z$ by
\begin{equation}\label{eq:cond-lowest}
Dz = z^{[1]}, \quad Bz = b(t)=:b_0(t).
\end{equation}
These conditions guarantee that the functions $A^{l-1}(f + k)$ for $1\le l\le m$ satisfy homogeneous boundary conditions.

\begin{proposition}\label{prob:superconvergence}
For sufficiently often differentiable data $f$ and $b$, the conditions \eqref{eq:cond-highest}--\eqref{eq:cond-lowest} imply that
\begin{equation}\label{eq:compatibility}
B A^{l-1}(f + k) = 0, \quad 1\le l\le m.
\end{equation}
Together with \eqref{eq:class-oc} this is a sufficient condition for the integrator to achieve order $s+m+1$.
\end{proposition}

\begin{proof}
The proof is carried out by induction. The case $m=1$ is proven in the paragraph preceding Theorem~\ref{thm:order-s+2}. For the general case, we assume that \eqref{eq:compatibility} holds for $1\le l\le m-1$. This implies that
$$
A^{m-1}(f+k) = D^{m-1}(f+k)
$$
for sufficiently often differentiable data. Therefore, we have
$$
B A^{m-1}(f + k) = B D^{m-1}(f + k) = B D^{m-1}f + B D^m z - B D^{m-1}z'.
$$
Now, the result follows from \eqref{eq:cond-highest} and \eqref{eq:cond-recursion} by using $D^l z = z^{[l]}$ for $l\ge 1$.
\end{proof}

In the remainder of this section, we will illustrate our theoretical considerations with some numerical experiments. We consider two problems of the form \eqref{linearproblem} with $\Omega = (0,1)$. This allows us to use fast Fourier techniques to compute the action of the operators $\varphi_j(\tau A)$ on vectors. We measure the errors pointwise in time using discrete versions of the $L^1$, $L^2$, and $L^{\infty}$ norms in space.  For example, given an equidistant discretization of $\Omega = (0,1)$ with mesh width $\Delta x = 1/(N+1)$, we use
\begin{equation*}
\|u\|_{2,\Delta x} = \textstyle\sqrt{\Delta x\sum_{i=1}^N |u_i|^2}
\end{equation*}
in the case of Dirichlet boundary conditions. For other norms or boundary conditions, this must be adapted accordingly.

\begin{example} \label{ex:e1}
Consider the following one-dimensional linear parabolic problem with time dependent Dirichlet boundary conditions:
\begin{equation}\label{eq:ex1}
\begin{alignedat}{2}
u_t & =  u_{xx} + \left(x^2 + x - 3\right)\text{e}^t, &&0 < x < 1, \quad 0 < t \leq 1,\\
u(0,x)  & =   x^2 + x, &&0 \leq x \leq 1, \\
u(t,0) & =  1 - \text{e}^t, \quad u(t,1) = 1 + \text{e}^t, \qquad\quad &&0 \leq t \leq 1.
\end{alignedat}
\end{equation}
We integrate this problem with the exponential quadrature rule \eqref{unumlinearscheme} based on the Gauss nodes with $s=2$ stages. The order conditions \eqref{orderconditions1} define such a method in a unique way after fixing the nodes. Moreover, the considered nodes satisfy the weak quadrature order conditions \eqref{eq:class-oc} up to classical order $m=4$.

We use the simple correction function $z(t,x) = 1 + (2x-1)\,\text{e}^t$ that satisfies the boundary conditions as required by Theorem~\ref{thm:beta}. A grid with $N = 512$ points and standard second order finite differences are used for the spatial discretization. To compute the errors, the exact solution $u(t,x) = 1 + (x^2 + x - 1)\, \text{e}^t$ is used. Due to the characterization of the domains of fractional powers of elliptic operators (see, e.g., \cite{Lofs}) the expected order of convergence is $3 + \beta$ with $\beta = \frac1{2p}$ for the norm $L^{p}$, $1\leq p \leq \infty$. The numerical results, collected in Table~\ref{tab:linear-beta1}, are in line with our theoretical results.

\begin{table}[ht]
\def\arraystretch{1.1}
\caption{Numerical results obtained by integrating \eqref{eq:ex1} with the exponential quadrature rule \eqref{unumlinearscheme} based on the Gauss nodes with $s=2$ stages. The error is measured at time $t=1$. \label{tab:linear-beta1}}
\begin{center}
\vspace{-4mm}
\small
\begin{tabular}{rrrrrrrrrr}
\multicolumn{1}{l}{ step size} & & \multicolumn{1}{l}{$L^{1}$ error\rule{0pt}{2.5ex}} & \multicolumn{1}{l}{order} & &
\multicolumn{1}{l}{$L^{2}$ error\rule{0pt}{2.5ex}} & \multicolumn{1}{l}{order} & &
\multicolumn{1}{l}{$L^{\infty}$ error\rule{0pt}{2.5ex}} & \multicolumn{1}{l}{order}\\
\hline
\rule{0pt}{11pt}
1.000e-01  & & 9.130e-06  & --         & & 1.036e-05  & --     & & 1.609e-04  & --          \\
5.000e-02  & & 8.385e-07  & 3.44       & & 1.117e-06  & 3.21   & & 2.059e-06  & 2.97       \\
2.500e-02  & & 7.427e-08  & 3.50       & & 1.171e-07  & 3.25   & & 2.574e-07  & 3.00       \\
1.250e-02  & & 6.619e-09  & 3.49       & & 1.229e-08  & 3.25   & & 3.220e-08  & 3.00        \\
6.250e-03  & & 5.866e-10  & 3.50       & & 1.292e-09  & 3.25   & & 4.025e-09  & 3.00        \\
\end{tabular}
\vspace{-5mm}
\end{center}
\end{table}

\end{example}

\begin{example}\label{ex:e2}
We consider the following one-dimensional linear parabolic problem with time invariant Dirichlet boundary conditions:
\begin{equation}\label{eq:ex2}
\begin{alignedat}{2}
u_t & =  u_{xx} + \left(-2 +12x -12x^2 + x^2(1-x)^2\right)\text{e}^t, \qquad &&0 < x < 1, \quad 0 < t \leq 1,\\
u(0,x)  & =   1 + x + x^2(1-x)^2, &&0 \leq x \leq 1, \\
u(t,0) & =  1, \quad u(t,1) = 2 , &&0 \leq t \leq 1,
\end{alignedat}
\end{equation}
We integrate this problem again with the exponential quadrature rule \eqref{unumlinearscheme} based on the Gauss nodes with $s=2$ stages. We consider two corrections $z$ with different behaviours. In both cases, we use a grid with $N = 512$ points, and standard second order finite differences are used for the spatial discretization. The errors are computed with respect to a reference solution obtained with the same numerical method using a very small time step $\tau = \frac1{4000}$.

We start with the simple correction function $z(t,x) = 1 + x$ that satisfies the boundary conditions, as required by Theorem~\ref{thm:beta}.  Due to the characterization of the domains of fractional powers of elliptic operators (see, e.g., \cite{Lofs}) the expected order of convergence is $3 + \beta$ with $\beta = \frac1{2p}$ for the norm $L^{p}$, $1\leq p \leq \infty$. The numerical results, collected in Table~\ref{tab:linear-beta}, are in line with our theoretical results.

\begin{table}[ht]
\def\arraystretch{1.1}
\caption{Numerical results obtained when integrating \eqref{eq:ex1} with the exponential quadrature rule \eqref{unumlinearscheme} based on the Gauss nodes with $s=2$ stages and the linear correction. The error is measured at time $t=1$.\label{tab:linear-beta}}
\begin{center}
\vspace{-4mm}
\small
\begin{tabular}{rrrrrrrrrr}
\multicolumn{1}{l}{ step size} & & \multicolumn{1}{l}{$L^{1}$ error\rule{0pt}{2.5ex}} & \multicolumn{1}{l}{order} & &
\multicolumn{1}{l}{$L^{2}$ error\rule{0pt}{2.5ex}} & \multicolumn{1}{l}{order} & &
\multicolumn{1}{l}{$L^{\infty}$ error\rule{0pt}{2.5ex}} & \multicolumn{1}{l}{order}\\
\hline
\rule{0pt}{11pt}
5.000e-02 & & 7.850e-07  & --         & & 9.201e-07  & --     & & 1.697e-06  & --        \\
2.500e-02  & & 7.940e-08  & 3.31       & & 1.054e-07  & 3.13   & & 2.344e-07  & 2.86     \\
1.250e-02  & & 7.373e-09  & 3.43       & & 1.163e-08  & 3.18   & & 3.075e-08  & 2.93     \\
6.250e-03 & & 6.554e-10  & 3.49       & & 1.254e-09  & 3.21   & & 3.936e-09  & 2.97      \\
3.125e-03  & & 5.707e-11  & 3.52       & & 1.334e-10  & 3.23   & & 4.963e-10  & 2.99      \\
\end{tabular}
\vspace{-5mm}
\end{center}
\end{table}

Next, we use the correction function $z(t,x) = 1 + (1-\text{e}^t)\, x + \text{e}^t\, x^2$ that satisfies the conditions required by Theorem~\ref{thm:order-s+2}. In this case, we expect convergence of order~$4$. The numerical results, given in Table~\ref{tab:linear-o4}, clearly confirm the expected order of convergence.

\begin{table}[ht]
\def\arraystretch{1.1}
\caption{Numerical results obtained by integrating \eqref{eq:ex2} with the exponential quadrature rule \eqref{unumlinearscheme} based on the Gauss nodes with $s=2$ stages and the quadratic correction. The error is measured at time $t=1$. \label{tab:linear-o4}}
\begin{center}
\vspace{-4mm}
\small
\begin{tabular}{rrrrrrrrrr}
\multicolumn{1}{l}{ step size} & & \multicolumn{1}{l}{$L^{1}$ error\rule{0pt}{2.5ex}} & \multicolumn{1}{l}{order} & &
\multicolumn{1}{l}{$L^{2}$ error\rule{0pt}{2.5ex}} & \multicolumn{1}{l}{order} & &
\multicolumn{1}{l}{$L^{\infty}$ error\rule{0pt}{2.5ex}} & \multicolumn{1}{l}{order}\\
\hline
\rule{0pt}{11pt}
5.000e-02 & & 5.292e-07  & --         & & 5.549e-07  & --     & & 6.407e-07  & --          \\
2.500e-02  & & 3.443e-08  & 3.94       & & 3.557e-08  & 3.96   & & 3.966e-08  & 4.01       \\
1.250e-02  & & 2.210e-09  & 3.96       & & 2.260e-09  & 3.98   & & 2.449e-09  & 4.02       \\
6.250e-03 & & 1.406e-10  & 3.97      & & 1.428e-10  & 3.98  & & 1.515e-10  & 4.01     \\
3.125e-03  & & 8.857e-12  & 3.99       & & 8.948e-12  & 4.00   & & 9.377e-12  & 4.01        \\
\end{tabular}
\vspace{-5mm}
\end{center}
\end{table}

\end{example}

\section{Semilinear problems}\label{sec:semilinear}

In this section we propose a strategy to approximate the solution of semilinear parabolic problems
\begin{align}\label{nonlinearproblem}
&\left\{
\begin{aligned}
\partial_t u & = Du + f(t,u), \\
B u & = b(t), \\
u(0) & = u_0.
\end{aligned}
\right.
\end{align}
We work with the same assumptions as in Section 2, but here the source term $f$ is also allowed to depend on the solution $u$. We first state our main assumption on the nonlinearity $f$.

\begin{assumption}\label{ass:lip}\rm
 Let $f : [0,T] \times X_{\alpha} \rightarrow X$ be locally Lipschitz. Thus there exists a real number $L(R,T)$ such that
\begin{equation}
\|f(t,u) - f(t,v)\| \leq L \|u - v\|_{\alpha},
\end{equation}
for all $t \in [0,T]$ and all $u$, $v$ with $\max \left( \|u\|_{\alpha}, \|v\|_{\alpha} \right) < R$.
\end{assumption}

For methods with convergence order higher than $p = 1$, we have to assume more regularity.

\begin{assumption}\label{ass:reg}\rm
We suppose that (\ref{nonlinearproblem}) possesses a sufficiently smooth solution $u : [0,T] \rightarrow X_{\alpha}$ with derivatives in $X_{\alpha}$, and that $f : [0,T] \times X_{\alpha} \rightarrow X$ is sufficiently often Fr\'{e}chet differentiable in a strip along the exact solution. All occurring derivatives are supposed to be uniformly bounded.
\end{assumption}

This framework allows us to consider a wide variety of models, for instance, reaction-diffusion problems, the incompressible Navier-Stokes equation or complex population models, see, e.g., \cite{Lunardi, Yagi}.

We again assume that we can compute a smooth function $z$ that solves the elliptic equation~(\ref{elliptic}). Then, we consider the function $w = u - z$, that is, the solution of the equation
\begin{align}\label{nonlinearcorrproblem}
&\left\{
\begin{aligned}
\partial_t w & = Dw + f(t,w+z) + k(t), \\
B w & = 0, \\
w(0) & = u_0 - z(0),
\end{aligned}
\right.
\end{align}
where $k = Dz - \partial_t z$. Note that once a solution $z$ of (\ref{elliptic}) is fixed, the new source term
$$
g(t,w) = f(t,w+z) + k(t)
$$
satisfies Assumption 2 whenever the function $f$ does. With this correction, our problem can be formulated in an abstract framework \cite{Henry, Pazy} as
\begin{align}\label{nonlinearproblemasbtract}
&\left\{
\begin{aligned}
w'(t) & = Aw(t) + g(t,w), \\
w(0) & = w_0.
\end{aligned}
\right.
\end{align}
For the solution of (\ref{nonlinearproblemasbtract}) we apply an explicit exponential Runge--Kutta method as in \cite{HO05e}, given~by
\begin{equation}\label{wnumsemischeme}
\begin{aligned}
w_{n+1} &= \text{e}^{\tau A} w_n + \tau \sum_{i=1}^s b_i\left(\tau A\right)G_{ni}, \\
W_{ni} &= \text{e}^{c_i \tau  A} w_n + \tau \sum_{j=1}^{i-1} a_{ij}\left(\tau A\right)G_{nj}, \\
G_{ni} &= f\bigl(t_n + c_i \tau, W_{ni} + z(t_n + c_i \tau)\bigr) + k(t_n+c_i\tau).
\end{aligned}
\end{equation}

The coefficients $a_{ij}, b_i$ in the previous scheme are designed to satisfy the so-called stiff order conditions (see Table 2 in \cite{HO05e}). Theorems~4.3 and 4.7 in \cite{HO05e} prove convergence of the methods satisfying these conditions under Assumptions 1-3. Here, we extend these results to the case of non-homogeneous boundary values by means of the correction previously explained. Before presenting the theorem, we express the numerical scheme \eqref{wnumsemischeme} in terms of $u$:
\begin{equation}\label{unumsemischeme}
\begin{aligned}
u_{n+1} &= \text{e}^{\tau A} \left(u_n - z(t_n)\right) + z(t_{n+1}) + \tau \sum_{i=1}^s b_i\left(\tau A\right)\,G_{ni}, \\
U_{ni} &= \text{e}^{\tau A} \left(u_n - z(t_n)\right) + z(t_n + c_i\tau) + \tau \sum_{j=1}^{i-1} a_{ij}\left(\tau A\right)\,G_{nj}, \\
G_{ni} &= f(t_n + c_i \tau, U_{ni})+ k(t_n + c_i \tau) .
\end{aligned}
\end{equation}
We are now ready to state our convergence result.

\begin{theorem}\label{thm:order-p}
Let the initial boundary value problem (\ref{nonlinearproblem}) satisfy Assumptions 1-3, and consider the solution given by the exponential Runge--Kutta method (\ref{unumsemischeme}). Then, the order of convergence of (\ref{unumsemischeme}) is the same as that of (\ref{wnumsemischeme}). In particular, if the stiff order conditions up to order $p$ are satisfied then convergence of order at least $p$ occurs.
\end{theorem}

\begin{proof}
Using
\begin{equation*}
\dfrac{\partial g}{\partial w} \left(t, w(t)\right) = \dfrac{\partial f}{\partial u} \left(t, u(t)\right)
\end{equation*}
and the fact that $w_n - w(t_n) = u_n - u(t_n)$, the application of Theorem 4.7 in \cite{HO05e} to (\ref{wnumlinearscheme}) leads to the desired result.
\end{proof}

We illustrate the previous result with some numerical examples. In all examples, we use the same three exponential Runge--Kutta methods, but make different choices for the correction $z$. In the following Butcher tableaux, we use the abbreviated notation
\begin{equation*}
\varphi_{j,k} = \varphi_j(c_k\tau A), \quad \varphi_j = \varphi_j(\tau A).
\end{equation*}
The first method we consider is the exponential version of the Euler method, with Butcher tableau:
\begin{equation}\label{tab:euler}
\begin{tabular}{c|c}
0 &  \\
\hline
 & $\varphi_1$ \\
\end{tabular}
\end{equation}
This method has order of convergence $p=1$. The second method is one of the proposed schemes by Strehmel and Weiner \cite{StrWei}, with tableau:
\begin{equation}\label{tab:sw}
\begin{tabular}{c|cc}
0 &  &  \\
$\frac{1}{2}$ & $\frac{1}{2}\, \varphi_{1,2}$ \\[3pt]
\hline
\rule{0pt}{11pt}& 0 & $\varphi_1$ \\
\end{tabular}
\end{equation}
We expect our experiments to show order $2$ with this method. The last method we consider is the following scheme proposed by Krogstad \cite{Krogstad}:
\begin{equation}\label{tab:krog}
\begin{tabular}{c|cccc}
0 &   &   &   &   \\
$\frac{1}{2}$ & $\frac{1}{2} \varphi_{1,2}$ &   &   &   \\[2pt]
$\frac{1}{2}$ & $\frac{1}{2} \varphi_{1,3} - \varphi_{2,3}$ & $\varphi_{2,3}$ &   &   \\[2pt]
1 & $\varphi_{1,4} - 2\varphi_{2,4}$ & 0 & $2 \varphi_{2,4}$ &   \\[2pt]
\hline
\rule{0pt}{11pt}& $\varphi_1 - 3 \varphi_2 + 4 \varphi_3$ & $2\varphi_2 - 4 \varphi_3$ & $2\varphi_2 - 4 \varphi_3$ & $-\varphi_2 + 4 \varphi_3$ \\
\end{tabular}
\end{equation}
This method has at least order $p=3$, but, as it is shown in \cite{HO05e}, it has even order $3 + \beta$, where $0 \leq \beta \leq 1$ is such that $\|A^{-1} J A^{\beta}\|$ is bounded and where
\begin{equation*}
J = \dfrac{\partial g}{\partial w}\left(t,w(t)\right).
\end{equation*}
In the examples below, the actions $\varphi_k(\tau A)u$ of matrix functions on vectors are computed using fast Fourier techniques.

\begin{example}\label{ex:e3}
We consider the following one-dimensional semilinear parabolic problem with time-dependent Dirichlet boundary conditions:
\begin{equation}\label{eq:ex3}
\begin{alignedat}{2}
u_t & = u_{xx} + u^2, &&0 \leq x \leq 1,\quad 0 \leq t \leq 0.5,\\
u(0,x)  & = 1 + \sin(\pi\left(x - 0.5\right)), &&0 \leq x \leq 1, \\
u(t,0) & = 1 - \exp(-\pi^2 t), \quad  u(t,1)  =  1 + \exp(-\pi^2 t),\qquad\quad &&0 \leq t \leq 0.5.
\end{alignedat}
\end{equation}
We approximate the solutions of the latter numerically by the scheme (\ref{unumsemischeme}) with different corrections $z$. In the first case, we use a correction $z$ satisfying $z_t = z_{xx}$, that is, $k = 0$,
\begin{equation}\label{eq:ex3-c1}
z(t,x) = 1 + \exp(-\pi^2 t) \sin(\pi(x-0.5)).
\end{equation}
A grid of $N = 512$ points and standard second order finite differences are used for the spatial discretization. The reference solution is computed with Krogstad's method and $\tau = \frac1{40000}$. In this example it happens that $\|A^{-1} J A\|$ is bounded, so we expect order 4 for Krogstad's method. The results, which are in accordance with expectations, are shown in Table~\ref{exp1tab1}.

\begin{table}[hb]
\def\arraystretch{1.1}
\caption{Errors and orders of convergence obtained by integrating \eqref{eq:ex3} with three different exponential Runge--Kutta methods of type~\eqref{unumsemischeme} using the correction~\eqref{eq:ex3-c1}. The errors are measured at time $t=0.5$. \label{exp1tab1}}
\begin{center}
\vspace{-2mm}
\small
\begin{tabular}{rrrrrrrrrr}
&& \multicolumn{2}{c}{Euler} && \multicolumn{2}{c}{Strehmel and Weiner} && \multicolumn{2}{c}{Krogstad} \\
\cline{3-4} \cline{6-7}  \cline{9-10}
\multicolumn{1}{c}{step size} && \multicolumn{1}{l}{\rule{0pt}{11pt}$L^{2}$ error} & \multicolumn{1}{r}{order} & &
\multicolumn{1}{l}{$L^{2}$ error} & \multicolumn{1}{r}{order} & &
\multicolumn{1}{l}{$L^{2}$ error} & \multicolumn{1}{r}{order}\\
\hline
\rule{0pt}{11pt}
5.000e-02  & & 2.121e-04  & --  & & 3.175e-05  & --  & & 2.435e-07  & --  \\
2.500e-02  & & 1.014e-04  & 1.06  & & 8.064e-06  & 1.98  & & 1.607e-08  & 3.92  \\
1.250e-02  & & 4.981e-05  & 1.03  & & 2.039e-06  & 1.98  & & 1.031e-09  & 3.96  \\
6.250e-03  & & 2.469e-05  & 1.01  & & 5.141e-07  & 1.99  & & 6.534e-11  & 3.98  \\
3.125e-03  & & 1.229e-05  & 1.01  & & 1.293e-07  & 1.99  & & 4.023e-12  & 4.02  \\
\end{tabular}
\vspace{-3mm}
\end{center}
\end{table}

In the second case, we use a correction $z$ satisfying $z_{xx} = 0$,
\begin{equation}\label{eq:ex3-c2}
z(t,x) = 1 + (2x-1)\exp(-\pi^2 t).
\end{equation}
We use the same grid, spatial discretization and reference solution. The results are given in Table \ref{exp1tab2}.

\begin{table}[t]
\def\arraystretch{1.1}
\caption{Errors and orders of convergence obtained by integrating \eqref{eq:ex3} with three different exponential Runge--Kutta methods of type~\eqref{unumsemischeme} using the correction~\eqref{eq:ex3-c2}. The errors are measured at time $t=0.5$.\label{exp1tab2}}
\begin{center}
\vspace{-3mm}
\small
\begin{tabular}{rrrrrrrrrr}
&& \multicolumn{2}{c}{Euler} && \multicolumn{2}{c}{Strehmel and Weiner} && \multicolumn{2}{c}{Krogstad} \\
\cline{3-4} \cline{6-7}  \cline{9-10}
\multicolumn{1}{c}{step size} && \multicolumn{1}{l}{\rule{0pt}{11pt}$L^{2}$ error} & \multicolumn{1}{r}{order} & &
\multicolumn{1}{l}{$L^{2}$ error} & \multicolumn{1}{r}{order} & &
\multicolumn{1}{l}{$L^{2}$ error} & \multicolumn{1}{r}{order}\\
\hline
\rule{0pt}{11pt}
2.500e-02  & & 2.420e-04  & --  & & 3.111e-05  & --  & & 5.854e-08  & --  \\
1.250e-02  & & 1.097e-04  & 1.14  & & 8.102e-06  & 1.94  & & 3.852e-09  & 3.93  \\
6.250e-03  & & 5.232e-05  & 1.07  & & 2.081e-06  & 1.96  & & 2.489e-10  & 3.95  \\
3.125e-03  & & 2.556e-05  & 1.03  & & 5.295e-07  & 1.97  & & 1.590e-11  & 3.97  \\
1.563e-03  & & 1.263e-05  & 1.02  & & 1.339e-07  & 1.98  & & 1.039e-12  & 3.94  \\
\end{tabular}
\vspace{-3mm}
\end{center}
\end{table}

For this example, we also present the errors in the norm of $X_{\alpha}$, with $\alpha = 0.5$. Recall that our convergence result for exponential methods, Theorem \ref{thm:order-p}, refers to convergence results in \cite{HO05e}. There, it is explained and numerically verified that when using the norm $X_{\alpha}$ with $\alpha = 0.5$, a reduction of up to $0.25$ in the order of convergence is expected for the method of Strehmel and Weiner, and for that of Krogstad, respectively, while the exponential Euler method is still expected to maintain order~1. The numerical results in Table~\ref{exp1tab3} are consistent with these expectations.

\begin{table}[t]
\def\arraystretch{1.1}
\caption{Errors and orders of convergence obtained by integrating \eqref{eq:ex3} with three different exponential Runge--Kutta methods of type~\eqref{unumsemischeme} using the correction~\eqref{eq:ex3-c2}. The errors measured at time $t=0.5$ in the norm of $X_{\alpha}$ for $\alpha = 0.5$.\label{exp1tab3}}
\begin{center}
\vspace{-3mm}
\small
\begin{tabular}{rrrrrrrrrr}
&& \multicolumn{2}{c}{Euler} && \multicolumn{2}{c}{Strehmel and Weiner} && \multicolumn{2}{c}{Krogstad} \\
\cline{3-4} \cline{6-7}  \cline{9-10}
\multicolumn{1}{c}{step size} && \multicolumn{1}{l}{\rule{0pt}{11pt}$H^{1}$ error} & \multicolumn{1}{r}{order} & &
\multicolumn{1}{l}{$H^{1}$ error} & \multicolumn{1}{r}{order} & &
\multicolumn{1}{l}{$H^{1}$ error} & \multicolumn{1}{r}{order}\\
\hline
\rule{0pt}{11pt}
2.500e-02  & & 1.787e-03  & --  & & 2.660e-04  & --  & & 4.686e-07  & --  \\
1.250e-02  & & 8.009e-04  & 1.16  & & 7.573e-05  & 1.81  & & 3.338e-08  & 3.81  \\
6.250e-03  & & 3.787e-04  & 1.08  & & 2.163e-05  & 1.81  & & 2.376e-09  & 3.81  \\
3.125e-03  & & 1.840e-04  & 1.04  & & 6.207e-06  & 1.80  & & 1.697e-10  & 3.81  \\
1.563e-03  & & 9.070e-05  & 1.02  & & 1.790e-06  & 1.79  & & 1.227e-11  & 3.79  \\
\end{tabular}
\vspace{-3mm}
\end{center}
\end{table}

\end{example}


\begin{table}[t]
\def\arraystretch{1.1}
\caption{Errors and orders of convergence obtained by integrating \eqref{eq:ex4} with three different exponential Runge--Kutta methods of type~\eqref{unumsemischeme} using the correction~\eqref{eq:ex4-c1}. The errors are measured at time $t=1$.\label{exp2tab1}}
\begin{center}
\vspace{-3mm}
\small
\begin{tabular}{rrrrrrrrrr}
&& \multicolumn{2}{c}{Euler} && \multicolumn{2}{c}{Strehmel and Weiner} && \multicolumn{2}{c}{Krogstad} \\
\cline{3-4} \cline{6-7}  \cline{9-10}
\multicolumn{1}{c}{step size} && \multicolumn{1}{l}{\rule{0pt}{11pt}$L^{2}$ error} & \multicolumn{1}{r}{order} & & \multicolumn{1}{l}{$L^{2}$ error} & \multicolumn{1}{r}{order} & & \multicolumn{1}{l}{$L^{2}$ error} & \multicolumn{1}{r}{order}\\
\hline
\rule{0pt}{11pt}
5.000e-02  & & 5.037e-02  & --  & & 2.753e-03  & --  & & 2.974e-07  & --  \\
2.500e-02  & & 2.627e-02  & 0.94  & & 7.175e-04  & 1.94  & & 2.200e-08  & 3.76  \\
1.250e-02  & & 1.343e-02  & 0.97  & & 1.830e-04  & 1.97  & & 1.719e-09  & 3.68  \\
6.250e-03  & & 6.792e-03  & 0.98  & & 4.621e-05  & 1.99  & & 1.407e-10  & 3.61  \\
3.125e-03  & & 3.416e-03  & 0.99  & & 1.161e-05  & 1.99  & & 1.216e-11  & 3.53  \\
\end{tabular}
\vspace{-3mm}
\end{center}
\end{table}

\begin{example}\label{ex:e4}
We now take into account the following one-dimensional semilinear parabolic problem with time-dependent oblique boundary conditions:
\begin{equation}\label{eq:ex4}
\begin{alignedat}{2}
u_t & = u_{xx} + u^2, &&0 \leq x \leq 1, \quad 0 \leq t \leq 1,\\
u(0,x)  & = 1 + \sin\bigl(0.5\,\pi\left(x - 1\right)\bigr), \qquad\quad&&0 \leq x \leq 1, \\
u(t,0) & = 1 - \exp\bigl(-0.25 \pi^2 t\bigr),  &&0 \leq t \leq 1,\\
u_x(t,1)  &=  \tfrac{\pi}2 \exp\bigl(-0.25 \pi^2 t\bigr), &&0 \leq t \leq 1.
\end{alignedat}
\end{equation}
We start with a first correction $z$ satisfying $z_t = z_{xx}$,
\begin{equation}\label{eq:ex4-c1}
z(t,x) = 1 + \exp(-0.25\pi^2 t) \sin\bigl(0.5\pi(x-1)\bigr).
\end{equation}
A grid of $N = 256$ points and standard second order finite differences are used for the spatial discretization. The reference solution is computed with Krogstad's method with the step size $\tau = \frac1{20000}$. In this example it happens that only $\|A^{-1} J A^{1/2}\|$ is bounded, so we expect order 3.5 for Krogstad's method. The results are shown in Table \ref{exp2tab1}.

In the second case, we use a correction $z$ satisfying $\Delta z = 0$,
\begin{equation}\label{eq:ex4-c2}
z(t,x) = 1 + (0.5\pi x-1)\exp\bigl(-0.25\pi^2 t\bigr).
\end{equation}
A grid of $N = 256$ points and standard second order finite differences are used in the spatial discretization. The reference solution is computed with Krogstad's method and step size $\tau = \frac1{20000}$. The results are given in Table \ref{exp2tab2}.

\begin{table}[ht]
\def\arraystretch{1.1}
\caption{Errors and orders of convergence obtained by integrating \eqref{eq:ex4} with three different exponential Runge--Kutta methods of type~\eqref{unumsemischeme} using the correction~\eqref{eq:ex4-c2}. The errors are measured at time $t=1$.\label{exp2tab2}}
\begin{center}
\vspace{-3mm}
\small
\begin{tabular}{rrrrrrrrrr}
&& \multicolumn{2}{c}{Euler} && \multicolumn{2}{c}{Strehmel and Weiner} && \multicolumn{2}{c}{Krogstad} \\
\cline{3-4} \cline{6-7}  \cline{9-10}
\multicolumn{1}{c}{step size} && \multicolumn{1}{l}{\rule{0pt}{11pt}$L^{2}$ error} & \multicolumn{1}{r}{order} & & \multicolumn{1}{l}{$L^{2}$ error} & \multicolumn{1}{r}{order} & & \multicolumn{1}{l}{$L^{2}$ error} & \multicolumn{1}{r}{order}\\
\hline
\rule{0pt}{11pt}
5.000e-02  & & 4.999e-02  & --  & & 2.634e-03  & --  & & 2.485e-07  & --  \\
2.500e-02  & & 2.608e-02  & 0.94  & & 6.829e-04  & 1.95  & & 1.740e-08  & 3.84  \\
1.250e-02  & & 1.334e-02  & 0.97  & & 1.736e-04  & 1.98  & & 1.374e-09  & 3.66  \\
6.250e-03  & & 6.746e-03  & 0.98  & & 4.376e-05  & 1.99  & & 1.161e-10  & 3.57  \\
3.125e-03  & & 3.393e-03  & 0.99  & & 1.098e-05  & 1.99  & & 9.423e-12  & 3.62  \\
\end{tabular}
\vspace{-3mm}
\end{center}
\end{table}

\end{example}


\begin{table}[t]
\def\arraystretch{1.1}
\caption{Errors and orders of convergence obtained by integrating \eqref{eq:ex5} with three different exponential Runge--Kutta methods of type~\eqref{unumsemischeme} and a numerically computed correction, see text. The errors are measured at time $t=0.5$.\label{exp3tab1}}
\begin{center}
\vspace{-3mm}
\small
\begin{tabular}{rrrrrrrrrr}
&& \multicolumn{2}{c}{Euler} && \multicolumn{2}{c}{Strehmel and Weiner} && \multicolumn{2}{c}{Krogstad} \\
\cline{3-4} \cline{6-7}  \cline{9-10}
\multicolumn{1}{c}{step size} && \multicolumn{1}{l}{\rule{0pt}{11pt}$L^{2}$ error} & \multicolumn{1}{r}{order} & &
\multicolumn{1}{l}{$L^{2}$ error} & \multicolumn{1}{r}{order} & &
\multicolumn{1}{l}{$L^{2}$ error} & \multicolumn{1}{r}{order}\\
\hline
\rule{0pt}{11pt}
1.250e-02  & & 1.334e-05  & --         & & 2.735e-07  & --     & & 9.499e-09  & --         \\
6.250e-03  & & 6.110e-06  & 1.13       & & 6.216e-08  & 2.14   & & 1.051e-09  & 3.17       \\
3.125e-03  & & 2.925e-06  & 1.06       & & 1.480e-08  & 2.07   & & 8.662e-11  & 3.60       \\
1.563e-03  & & 1.431e-06  & 1.03       & & 3.634e-09  & 2.03   & & 6.167e-12  & 3.85       \\
7.813e-03  & & 7.079e-07  & 1.02       & & 9.332e-10  & 2.01   & & 5.608e-13  & 3.94       \\
\end{tabular}
\vspace{-3mm}
\end{center}
\end{table}

\begin{example} \label{ex:e5}
We consider the domain $\Omega = (0,1)\times(0,1)$ and study the two-dimensional semilinear parabolic problem subject to Dirichlet boundary conditions:
\begin{equation}\label{eq:ex5}
\begin{alignedat}{3}
u_t & = \Delta u + u^2, \qquad\quad && (x,y) \in \Omega, &&0 \leq t \leq 0.5, \\
u(0,x,y) & = f(x,y), &&(x,y) \in \Omega, \\
u(t,x,y) & = f(x,y), &&(x,y) \in \partial\Omega,\quad &&0 \leq t \leq 0.5,
\end{alignedat}
\end{equation}
where
\begin{align*}
f(x,y) &=  0.5 + 2\exp\bigl(-40(x-0.5-0.1\cos^2 \pi y)\bigr) + 2\exp\bigl(-35(y-0.5-0.1\sin^2 2\pi x)\bigr) \\
&\quad -  2\exp\bigl(-35 \bigl((x-0.5)^2 + (y-0.5)^2\bigr)\bigr).
\end{align*}
In this problem, $z$ is computed by numerically solving the Laplace equation with the prescribed boundary conditions. A grid of $N = 128$ points per direction and standard second order finite differences are used for the spatial discretization. The reference solution is computed with the explicit Runge--Kutta method RK4 and step size $\tau = 10^{-5}$. We expect again order~$4$ for Krogstad. The results are shown in Table \ref{exp3tab1}.

\end{example}

\section{Conclusions}\label{sec:conclude}

The standard construction of exponential Runge--Kutta methods relies on semigroup theory, which only covers homogeneous and periodic boundary conditions. Direct application of such integrators to problems with non-homogeneous boundary conditions typically results in a reduction in order and loss of accuracy. In this paper, we proposed a simple modification that restores the convergence order. The required solutions to stationary elliptic problems can be obtained using standard techniques, which are also employed in splitting methods. Because the additional computational cost is negligible, the proposed approach yields highly competitive schemes. Similar modifications can be made for other types of integrators, such as implicit Runge--Kutta methods, (exponential) Rosenbrock methods, and (exponential) Lawson methods.

\section{Acknowledgments}
Carlos Arranz-Simón is supported by the Spanish Ministerio de Ciencia e Innovación through Project PID2023-147073NB-I00.



\end{document}